\documentclass[a4paper,fleqn]{cas-dc}

\usepackage[numbers,sort&compress]{natbib}
\usepackage{mathrsfs}
\usepackage{amsmath}
\usepackage{amsfonts}
\usepackage{amssymb}
\usepackage{amsthm}
\usepackage{enumerate}
\usepackage{setspace}
\usepackage{bbm}
\usepackage{color}
\usepackage{graphicx}
\usepackage[nameinlink,noabbrev]{cleveref}
\hypersetup{
    colorlinks=true,
    linkcolor=blue,
    citecolor=red,
    urlcolor=blue
}

\makeatletter
\renewcommand{\fnum@figure}{Fig. \thefigure.\@gobble}
\makeatother

\newtheorem{theorem}{Theorem}[section]
\newtheorem{proposition}[theorem]{Proposition}
\newtheorem{lemma}[theorem]{Lemma}
\newtheorem{remark}[theorem]{Remark}

\crefname{figure}{Fig.}{Figs.}
\crefname{subsection}{Subsection}{Subsection}
\crefname{theorem}{Theorem}{Theorems}
\crefname{example}{Example}{Examples}
\crefname{proposition}{Proposition}{Propositions}
\crefname{corollary}{Corollary}{Corollaries}
\crefname{definition}{Definition}{Definitions}
\crefname{lemma}{Lemma}{Lemmas}
\crefname{remark}{Remark}{Remarks}
\crefname{assumption}{Assumption}{Assumptions}
\crefname{problem}{Problem}{Problems}

\begin{document}

\shorttitle{Mean-variance portfolio selection: mean-field and forward-backward systems with initial-terminal coupling}
\shortauthors{X. Li et al.}

\title[mode=title]{Mean-variance portfolio selection: connections between mean-field systems and forward-backward stochastic systems under initial-terminal cost coupling}

\author[1]{Xun Li}
\credit{Conceptualization, Formal analysis, Methodology, Validation, Writing-review \& editing}

\author[2]{Hongyu Shi}[orcid=0009-0008-8371-5103]
\credit{Conceptualization, Formal analysis, Funding acquisition, Investigation, Methodology, Validation, Writing-original draft, Writing-review \& editing}
\cormark[1]

\author[2]{Zhen Wu}
\credit{Conceptualization, Funding acquisition, Project administration, Supervision, Writing-review \& editing}

\cortext[1]{Corresponding author.\\
\textit{E-mail address:}
\href{mailto:202112026@mail.sdu.edu.cn}
{202112026@mail.sdu.edu.cn} (H. Shi).}

\affiliation[1]{
  organization={Hong Kong Polytechnic University},
  addressline={Clear Water Bay},
  city={Kowloon},
  state={Hong Kong},
  country={China}
}

\affiliation[2]{
  organization={Shandong University},
  addressline={Shanda South Road No. 27},
  city={Jinan},
  postcode={250100},
  state={Shandong},
  country={China}
}

\begin{abstract}
This paper investigates the connection between mean-field systems and forward-backward stochastic systems arising from the mean-variance portfolio selection problem. By introducing the conditional expectation process of terminal wealth, the mean-field term is represented by a backward state, leading to a class of control problems for forward-backward stochastic systems with coupled initial-terminal values. The coupling between the terminal value of the forward state and the initial value of the backward state in the cost functional prevents classical methods for forward-backward stochastic optimal control from being directly applicable. To overcome this difficulty, we establish a stochastic maximum principle for control problems with running costs and coupled initial-terminal values, and provide sufficient conditions for optimality. Furthermore, we investigate a class of generalized stochastic linear-quadratic control problems with initial-terminal coupling, allowing the control weight in the running cost to be indefinite, and derive the optimal feedback representation via a Riccati equation. The classical mean-variance portfolio selection model is shown to be a special case, for which the corresponding efficient investment strategy and efficient frontier are obtained explicitly, and numerical simulations are presented to illustrate the results and examine parameter sensitivity and its financial implications.
\end{abstract}

\begin{highlights}
  \item Linking mean-field systems and forward-backward stochastic systems.
  \item Conditional expectation converts mean-field terms into a backward state.
  \item Optimal control with initial-terminal coupling in the cost functional.
  \item Maximum principles for forward-backward stochastic control with coupled costs.
  \item Indefinite stochastic linear-quadratic control for mean-variance portfolio selection.
\end{highlights}


\maketitle

\section{Introduction}\label{sec1}

Markowitz \cite{M1952} initiated mean-variance portfolio theory. With the development of financial theory and stochastic control, dynamic mean-variance portfolio selection has gradually developed into a systematic research framework \cite{M1969,LN2000,ZL2000}, and has been further extended to settings with random coefficients and more general market structures \cite{NLZK2020,WLXZ2023,CCK2024}. For the classical mean-variance model, suppose that $d + 1$ assets  are continuously traded in the market. One of the assets is a bond, whose price process $P^0_{\cdot}$ satisfies the following ordinary differential equation:
\begin{equation}
  \label{eq001}
  \left\{
    \begin{aligned}
      &\mathrm{d}P^0_t=r_tP^0_t\mathrm{d}t,\quad t\in[0,T],\\
      &P^0_0 = p_0>0,
    \end{aligned}
  \right.
\end{equation}
where $r_{\cdot} > 0$ denotes the risk-free interest rate associated with the bond. The remaining $d$ assets are stocks, whose price processes $P^1_{\cdot},P^2_{\cdot},$ $\cdots,P^d_{\cdot}$ satisfy the following stochastic differential equations (SDEs):
\begin{equation}
  \label{eq002}
  \left\{
    \begin{aligned}
      &\mathrm{d}P^i_t = P^i_t\left(b^i_t\mathrm{d}t+\sum^d_{j=1}\sigma^{ij}_t\mathrm{d}W^j_t\right),\quad t\in[0,T],\\
      &P^i_0 = p_i>0,
    \end{aligned}
  \right.
\end{equation}
where $b^i_{\cdot}$ denotes the appreciation rate of the $i$th stock, and $\sigma_{\cdot}:=\left(\sigma^{ij}_{\cdot}\right)_{d\times d}$ denotes the volatility matrix of the stocks. We assume that $r_{\cdot},b^i_{\cdot},\sigma_{\cdot}$ are uniformly bounded measurable functions and that $\sigma_{\cdot}$ is nondegenerate, i.e., $\sigma_{\cdot}\sigma^{\top}_{\cdot}\gg 0$. Consider an investor who holds $N^i_t$ units of the $i$th asset at time $t$, where $i = 0, 1,\cdots, d$. Combining \eqref{eq001} and \eqref{eq002}, the investor's total wealth process $x_{\cdot}$ satisfies
\begin{equation}
  \label{eq003}
  \left\{
    \begin{aligned}
      &\mathrm{d}x_t = \sum^d_{i=0} N^i_t\mathrm{d}P^i_t= \left\{r_tx_t+\sum^d_{i=1}(b^i_t-r_t)u^i_t\right\}\mathrm{d}t\\
      &\qquad\qquad\qquad\qquad+\sum^d_{i=1}\sum^d_{j=1}\sigma^{ij}_tu^i_t\mathrm{d}W^j_t,t\in[0,T],\\
      &x_0 = a>0,
    \end{aligned}
  \right.
\end{equation}
where $u_{\cdot}^i=N^i_{\cdot}P^i_{\cdot}$, $i=0,1,\cdots,d$, denotes the total market value invested in the $i$th asset. We call $u_{\cdot} = (u^1_{\cdot},u^2_{\cdot},\cdots,u^d_{\cdot})^{\top}$ the investor's portfolio. The investor aims to maximize the mean terminal wealth $\mathbb{E}x_T$ while minimizing the variance of terminal wealth $Var [x_T] = \mathbb{E}[x_T-\mathbb{E}x_T]^2$. This is a multi-objective optimization problem involving two conflicting criteria. The objective can be formulated as a weighted combination of the mean and variance. Specifically, subject to $(x_{\cdot},u_{\cdot})$ satisfying \eqref{eq003}, the mean-variance portfolio selection problem can be formulated as minimizing
\begin{equation}
  \label{eq004}
  Var[x_T]-\lambda \mathbb{E}x_T = \mathbb{E}\big[(x_T-\mathbb{E}x_T)^2\big]-\lambda \mathbb{E}x_T,
\end{equation}
where $\lambda \geq 0$ is a parameter. The optimal portfolio $\bar{u}_{\cdot}$ is called an efficient strategy, and the corresponding $(\mathbb{E}\bar{x}_T, Var[\bar{x}_T])$ is called an efficient point. As $\lambda$ varies over $[0, +\infty)$, the set of all efficient points forms the efficient frontier. Zhou and Li \cite{ZL2000} transformed the continuous-time mean-variance problem into an auxiliary stochastic linear-quadratic (LQ) control problem by means of an embedding technique; Li, Zhou, and Lim \cite{LZL2002} employed stochastic LQ control together with the Hamilton--Jacobi--Bellman equation and viscosity solution theory to derive the efficient strategy and efficient frontier. On the other hand, due to the time inconsistency induced by the mean term, equilibrium control has become an important line of research for dynamic mean-variance problems \cite{HJ2022,ZJWY2022}. Meanwhile, backward stochastic differential equation (BSDE) approaches and data-driven methods have also been applied to such problems \cite{ZL2023,WL2024}.

Combining \eqref{eq003} and \eqref{eq004}, one can see that the portfolio optimization problem is essentially a mean-field control problem for a forward stochastic system. The stochastic maximum principle (SMP) and dynamic programming theory for this class of problems have been developed in fairly general settings \cite{AD2011,BDL2011,ABC2019,PW2017,DPT2022}. In the LQ setting, Riccati equations have been used to characterize optimal controls and their feedback forms, and further developments include open-loop and closed-loop solvability as well as control problems with random coefficients \cite{Y2013,LSY2016,S2017,SW2021,XX2025}. It is worth noting, however, that the mean-field term $[\mathbb{E}x_T]^2$ in \eqref{eq004} makes the objective functional nonseparable in the sense of dynamic programming, which makes the problem difficult to solve directly (see \cite{ZL2000}). To eliminate the mean-field term from the optimization objective, we set $y_t := \mathbb{E}[x_T | \mathcal{F}_t], 0\leq t \leq T$. It is clear that $y_{\cdot}$ is an $\mathcal{F}_t$-martingale and that $y_0 = \mathbb{E}x_T$. Hence, by the martingale representation theorem, together with \eqref{eq003} and \eqref{eq004}, the forward mean-field control problem associated with the mean-variance model can be transformed into the following control problem for a forward-backward stochastic system:
\begin{equation}
  \label{eq005}
  \begin{aligned}
    & \text{Minimize }\text{ } \mathbb{E}\big[(x_T-y_0)^2-\lambda x_T\big],\\
    & \text{subject to }
    \left\{
      \begin{aligned}
      &\mathrm{d}x_t = \left\{r_tx_t+\sum^d_{i=1}(b^i_t-r_t)u^i_t\right\}\mathrm{d}t\\
      &\qquad\qquad+\sum^d_{i=1}\sum^d_{j=1}\sigma^{ij}_tu^i_t\mathrm{d}W^j_t,\\
      &\mathrm{d}y_t = \sum^d_{j=1}z^{j}_t\mathrm{d}W^j_t,\\
      &x_0 = a, \quad y_T = x_T.
    \end{aligned}
    \right.
  \end{aligned}
\end{equation}

Results on the well-posedness of forward-backward stochastic differential equations (FBSDEs) and the dependence of their solutions on parameters have been established \cite{PP1992,W1998a}. Stochastic control theory for FBSDEs has also been extended to more general settings, including general control domains and fully coupled systems \cite{W1998,SW2006,HJX2018}. In the LQ framework, solvability conditions can be further established by means of the corresponding Riccati-type equations and decoupling methods, and explicit expressions for the optimal controls can be obtained \cite{HLS2012,WWX2015,HJX2023}. In addition, such models can also be applied to problems in insurance and financial decision-making \cite{SW2010,HWW2010,WZ2018,WYZ2024}. Inspired by the above motivations, this paper studies the dynamics of forward-backward systems with coupled mixed conditions. Specifically, we consider a class of control problems for forward-backward systems with coupled initial-terminal values, where the cost functional is given by
\begin{equation*}
  J(u_{\cdot}) = \mathbb{E}\left[\int^T_0f(t,x_t,y_t,z_t,u_t)\mathrm{d}t+\Phi(x_T,y_0)\right].
\end{equation*}
However, as can be seen from \eqref{eq005}, our model not only involves the case of an indefinite stochastic LQ control problem, but also contains a fully coupled term between the terminal value $x_T$ of the forward process and the initial value $y_0$ of the backward process. This renders the traditional methods for forward-backward stochastic optimal control problems no longer applicable; see \cite{W1998,SW2006,P1993}. Therefore, the research content of this paper is highly challenging.

The main contributions of this paper are summarized as follows:

\noindent (i) Starting from the mean-variance portfolio selection problem, this paper introduces the conditional expectation process of terminal wealth and represents the mean-field term in the objective functional as a backward state. In turn, a connection between mean-field systems and forward-backward stochastic systems is established, and the coupling between the terminal value of the forward state and the initial value of the backward state induced by the mean term is revealed.

\noindent (ii) We study a class of control problems for forward-backward systems with coupled initial-terminal values. By introducing the corresponding variational equation and adjoint equations, we derive a stochastic maximum principle (SMP) and further establish sufficient conditions for optimality under appropriate convexity conditions. The result applies to a general cost functional with a running cost, rather than being restricted to the case where the cost contains only initial-terminal terms.

\noindent (iii) This paper further studies a class of generalized stochastic LQ control problems whose cost functional contains both a running cost and an initial-terminal coupling term, while relaxing the uniform positive definiteness condition on the control weight $R_{\cdot}$. By means of a Riccati equation, we derive a feedback representation of the optimal control. The classical mean-variance model is included as a special case of this framework, for which the efficient investment strategy, efficient frontier, and corresponding numerical results are obtained, with sensitivity analysis revealing how parameter changes affect optimal investment decisions.

The remainder of this paper is organized as follows. We first study a class of control problems for forward-backward systems with coupled initial-terminal values and establish the corresponding SMP in \Cref{sec2}. The analysis then turns to generalized stochastic LQ control problems in \Cref{sec3}, where the results are applied to the classical mean-variance portfolio selection problem. The paper concludes with \Cref{sec4}.

\section{A control problem with coupled initial-terminal values}\label{sec2}
Let $T>0$ and $(\Omega, \mathcal{F}, \mathbb F, P)$ be a filtered complete probability space. On this space, there is a $d$-dimensional standard Brownian motion $W_{\cdot}=(W_{\cdot}^1,W_{\cdot}^2,\cdots,W_{\cdot}^d)^{\top}$. $\mathbb{F}={{\left\{ {{\mathcal{F}}_{t}} \right\}}_{0\le t\le T}}$ is a natural filtration generated by $\{W_s,0\leq s\leq t\}$ and augmented by all $\mathbb{P}$-null sets. $\mathcal G$ is the progressive $\sigma$-field on $\Omega\times [0,T]$. $\mathbb{R}^{m\times n}$ is the $(m\times n)$-dimensional Euclidean space. $|\cdot|$ and $\langle \cdot,\cdot\rangle$ denote the norm and scalar product in the Euclidean space. $D^{\top}$ ($D^{-1}$) denotes the transpose (inverse) of matrix $D$. $\mathbb S^n$ denotes the set of symmetric $n\times n$ matrices with real elements. If $D\in \mathbb{S}^n$ is positive definite (positive semi-definite, uniformly positive definite), we write $D>(\geq,\gg)0$. For a given Hilbert space $\mathbb H$, $C([0,T];\mathbb H)$ denotes the space of all $\mathbb H$-valued continuous functions $\eta_{\cdot}$ on $[0, T]$ such that $\max_{0\leq t\leq T}|\eta_t| < \infty$; $L^2(\Omega;\mathbb H)$ denotes the space of all $\mathbb H$-valued random variables $\zeta$ such that $\mathbb{E}[|\zeta|^2] < \infty$; $L^2_{\mathcal{F}_T}(\mathbb{H})$ denotes the set of all ${\mathcal{F}_T}$ measurable random variables $\vartheta \in L^2(\Omega;\mathbb H)$; $S^2_{\mathbb{F}}(0,T;\mathbb{H})$ denotes the space of all $\mathbb H$-valued, $\mathcal{F}_t$-adapted c\`adl\`ag processes $\phi_{\cdot}$ such that $ \mathbb{E}[\sup_{0\leq t\leq T}|\phi_t|^2] < \infty$; $L^2_{\mathbb{F}}(0,T;\mathbb{H})$ denotes the space of all $\mathbb H$-valued, $\mathcal G$-measurable processes $\xi_{\cdot}$ such that $ \mathbb{E}[\int_{0}^{T}|\xi_t|^2\mathrm{d}t] < \infty$. 

Next, we investigate a class of control problems for forward-backward systems with coupled initial-terminal values, of which \eqref{eq005} is a special case. Let $U \subseteq \mathbb{R}^k$ be non-empty convex set. We consider the forward-backward system:
\begin{equation}
  \label{eq006}
  \left\{
    \begin{aligned}
      &\mathrm{d}x_t = b(t,x_t,u_t)\mathrm{d}t + \sigma (t,x_t,u_t)\mathrm{d}W_t,\\
      &\mathrm{d}y_t = -g(t,x_t,y_t,z_t,u_t)\mathrm{d}t + z_t\mathrm{d}W_t,\\
      &x_0 = a,\quad y_T = h(x_T),
    \end{aligned}
  \right.
\end{equation}
where $(b,\sigma):\Omega \times [0,T]\times {{\mathbb{R}}^{n}}\times  U \to {{\mathbb{R}}^{n}}\times {{\mathbb{R}}^{n\times d}}$, $g:\Omega \times [0,T]\times {{\mathbb{R}}^{n}}\times \mathbb{R}^m \times \mathbb{R}^{m\times d} \times  U \to {{\mathbb{R}}^{m}}$, $h:\Omega\times \mathbb R^n\rightarrow \mathbb R^m$ and $a \in \mathbb{R}^n$. Define the admissible control set $\mathcal U_{ad}:=\{u| u_{\cdot} \in L^2_{\mathbb{F}}(0,T;U)\}$. Then we give some assumptions to ensure the well-posedness of the control system as follows.

\noindent \textbf{Assumption (H1)} i) $(b,\sigma)$ is $\mathcal{G}\otimes \mathcal{B}({{\mathbb{R}}^{n}})\otimes{{\mathcal{U}}_{ad}}/\mathcal{B}({{\mathbb{R}}^{n}})\otimes \mathcal{B}({{\mathbb{R}}^{n\times d}})$ measurable. $g$ is $\mathcal{G}\otimes \mathcal{B}({{\mathbb{R}}^{n}})\otimes \mathcal{B}({{\mathbb{R}}^{m}})\otimes \mathcal{B}({{\mathbb{R}}^{m\times d}})\otimes{{\mathcal{U}}_{ad}}/\mathcal{B}({{\mathbb{R}}^{m}})$ and $h$ is $\mathcal F_T\otimes \mathcal B(\mathbb R^n)/\mathcal B(\mathbb R^m)$ measurable.

\noindent ii) $b,\sigma,g,h$ are continuously differentiable w.r.t $(x,y,z,u)$, their partial derivatives are uniformly bounded and continuous w.r.t $(x,y,z,u)$.

\noindent iii) $E[(\int_0^T\!(|b(t,0,0)|+|g(t,0,0,0,0)|)\mathrm{d}t)^2+\int_0^T\!|\sigma(t,0,0)|^2\mathrm{d}t$ $+|h(0)|^2]<\infty$.

Under (H1), for admissible control $u_{\cdot}\in \mathcal U_{ad}$, it is easy to check that FBSDE \eqref{eq006} admits a unique solution $(x_{\cdot},y_{\cdot},z_{\cdot}) \in S^2_{\mathbb{F}}(0,T;\mathbb{R}^{n})\times S^2_{\mathbb{F}}(0,T;\mathbb{R}^{m})\times L^2_{\mathbb{F}}(0,T;\mathbb{R}^{m\times d})$. Next, we introduce the cost functional with coupled initial-terminal values as follows.
\begin{equation}
  \label{eq007}
  J(u_{\cdot}) = \mathbb{E}\left[\int^T_0f(t,x_t,y_t,z_t,u_t)\mathrm{d}t+\Phi(x_T,y_0)\right],
\end{equation}
where $f:\Omega \times [0,T]\times {{\mathbb{R}}^{n}}\times \mathbb{R}^m \times \mathbb{R}^{m\times d} \times  U \to \mathbb{R}$, $\Phi:\Omega\times \mathbb R^n\times \mathbb R^m\rightarrow \mathbb{R}$. Here we introduce some assumptions.

\noindent \textbf{Assumption (H2)} i) $f$ is $\mathcal{G}\otimes \mathcal{B}({{\mathbb{R}}^{n}})\otimes \mathcal{B}({{\mathbb{R}}^{m}})\otimes \mathcal{B}({{\mathbb{R}}^{m\times d}})\otimes{{\mathcal{U}}_{ad}}/\mathcal{B}({{\mathbb{R}}})$ and $\Phi$ is $\mathcal F_T\otimes \mathcal B(\mathbb R^n)\otimes \mathcal B(\mathbb R^m)/\mathcal B(\mathbb R)$ measurable.

\noindent ii) $f$ and $\Phi$ are continuously differentiable w.r.t $(x,y,z,u)$, their partial derivatives are continuous w.r.t $(x,y,z,u)$.

\noindent iii) There exists a constant $C>0$ such that $|f(t,x,y,z,u)|\leq C (1+|x|^2+|y|^2+|z|^2+|u|^2)$, $|(f_x,f_y,f_z,f_u)(t,x,y,z,u)|\leq C (1+|x|+|y|+|z|+|u|)$, $|\Phi(x,y)| \leq C(1+|x|^2+|y|^2)$, $|(\Phi_{x},\Phi_{y})(x,y)|\leq C(1+|x|+|y|)$.

Now, the optimal control problems for forward-backward systems with coupled initial-terminal values can be described as follows.

\noindent \textbf{Problem (FBIT)} Find an admissible control $\bar u_{\cdot}\in \mathcal U_{ad}$ such that
\begin{equation*}
  J(\bar u_{\cdot}) = \inf_{u_{\cdot}\in \mathcal U_{ad}} J(u_{\cdot}).
\end{equation*}

\subsection{Variational analysis}
\label{subsec3_1}
In this part, we derive the variational inequality for the cost functional with coupled initial-terminal values \eqref{eq007}. Let $\bar u_{\cdot}$ be an optimal control for the Problem (FBIT), and the corresponding  trajectories are denoted by $(\bar x_{\cdot},\bar y_{\cdot},\bar z_{\cdot})$. Let $v_{\cdot}\in \mathcal U_{ad}$ be another admissible control. Then $u^{\epsilon}_{\cdot}:=(1-\epsilon)\bar u_{\cdot}+\epsilon v_{\cdot}\in \mathcal U_{ad}$ for $0\leq \epsilon\leq 1$ because of the convexity of the admissible control set and the corresponding  trajectories are denoted by $(x^{\epsilon}_{\cdot},y^{\epsilon}_{\cdot}, z^{\epsilon}_{\cdot})$. Next, we introduce the following variational equation:
\begin{equation}
  \label{eq008}
  \left\{
    \begin{aligned}
      &\mathrm{d}x^1_t = \big(b_t^xx^1_t+b_t^u\delta \bar u_t\big)\mathrm{d}t + \big(\sigma_t^xx^1_t+\sigma_t^u\delta \bar u_t\big)\mathrm{d}W_t,\\
      &\mathrm{d}y^1_t = -\big(g_t^xx^1_t+g_t^yy^1_t+g_t^zz^1_t+g_t^u\delta \bar u_t\big)\mathrm{d}t + z^1_t\mathrm{d}W_t,\\
      &x^1_0 = 0, \quad y^1_T = h_T^xx^1_T,
    \end{aligned}
  \right.
\end{equation}
where $\delta \bar u_{\cdot}=v_{\cdot}-\bar u_{\cdot}$, $b^{\tau}_{\cdot}=b_{\tau}({\cdot},\bar x_{\cdot},\bar u_{\cdot})$, $\sigma^{\tau}_{\cdot}=\sigma_{\tau}({\cdot},\bar x_{\cdot},\bar u_{\cdot})$, $g^{\tau}_{\cdot}=g_{\tau}({\cdot},\bar x_{\cdot},\bar y_{\cdot},\bar z_{\cdot},\bar u_{\cdot})$, $h^{x}_T=h_{x}(\bar x_T)$ with $\tau =x,y,z,u$. Under (H1), for admissible control $v_{\cdot}\in \mathcal U_{ad}$, it is easy to see that FBSDE \eqref{eq008} admits a unique solution $(x^1_{\cdot},y^1_{\cdot},z^1_{\cdot}) \in S^2_{\mathbb{F}}(0,T;\mathbb{R}^{n})\times S^2_{\mathbb{F}}(0,T;\mathbb{R}^{m})\times L^2_{\mathbb{F}}(0,T;\mathbb{R}^{m\times d})$. 

For simplicity, we set $(\Delta x^{\epsilon}_{\cdot},\Delta y^{\epsilon}_{\cdot},\Delta z^{\epsilon}_{\cdot}):=(\frac{x^{\epsilon}_{\cdot}-\bar x_{\cdot}}{\epsilon}-x_{\cdot}^1, \frac{y^{\epsilon}_{\cdot}-\bar y_{\cdot}}{\epsilon}-y_{\cdot}^1,\frac{z^{\epsilon}_{\cdot}-\bar z_{\cdot}}{\epsilon}-z_{\cdot}^1)$. In order to get the variational equation for the cost functional, we require the following lemma. The detailed proof can be found in Lemma 4.1 of \cite{P1993}.
\begin{lemma}
  \label{le3_1}
  Under (H1), we have
  \begin{equation*}
    \begin{aligned}
      &\underset{\epsilon \to 0}{\mathop{\lim }}\,\sup_{0\leq t\leq T}\mathbb{E}|\Delta x^{\epsilon}_t|^2 = 0,\quad \underset{\epsilon \to 0}{\mathop{\lim }}\,\sup_{0\leq t\leq T}\mathbb{E}|\Delta y^{\epsilon}_t|^2 = 0,\\
      &\underset{\epsilon \to 0}{\mathop{\lim }}\,\mathbb{E}\int^T_0|\Delta z^{\epsilon}_t|^2\mathrm{d}t = 0.
    \end{aligned}
  \end{equation*}
\end{lemma}

With the help of \Cref{le3_1}, we can obtain the following variational inequality for the cost functional with coupled initial-terminal values \eqref{eq007}. To simplify notation, we set $f^{\tau}_{\cdot}=f_{\tau}({\cdot},\bar x_{\cdot},\bar y_{\cdot},\bar z_{\cdot},\bar u_{\cdot})$, $\Phi^{x}_T=\Phi_{x}(\bar x_T,\bar y_0), \Phi^{y}_0=\Phi_{y}(\bar x_T,\bar y_0)$ with $\tau =x,y,z,u$.
\begin{lemma}
  \label{le3_2}
  Under (H1)-(H2), we have
  \begin{equation*}
    \begin{aligned}
      &\quad \text{ }\underset{\epsilon \to 0}{\mathop{\lim }}\,\frac{J(u^{\epsilon}_{\cdot})-J(\bar u_{\cdot})}{\epsilon } \\
      &=\mathbb{E}\bigg[\int^T_0\big((f^x_t)^{\top}x^1_t+(f^y_t)^{\top}y^1_t+(f^z_t)^{\top}z^1_t+(f^u_t)^{\top}\delta \bar u_t\big)\mathrm{d}t\\
      &\qquad+(\Phi^{x}_T)^{\top}x^1_T+(\Phi^{y}_0)^{\top}y^1_0\bigg]\geq 0.
    \end{aligned}
  \end{equation*}
\end{lemma}
\begin{proof}
  Recalling that the cost functional is defined by \eqref{eq007}, we have
  \begin{equation*}
    \begin{aligned}
      &\quad \text{ }J(u^{\epsilon}_{\cdot})-J(\bar u_{\cdot})\\
      &= \mathbb{E}\bigg[\int^T_0\big(f(t,x^{\epsilon}_t,y^{\epsilon}_t,z^{\epsilon}_t,u^{\epsilon}_t)-f(t,\bar x_t,\bar y_t,\bar z_t,\bar u_t)\big)\mathrm{d}t\\
      &\qquad+\Phi(x^{\epsilon}_T,y^{\epsilon}_0)-\Phi(\bar x_T,\bar y_0)\bigg].
    \end{aligned}
  \end{equation*}
  Following the proof of Lemma 4.2 in \cite{P1993}, it is easy to see that
  \begin{equation*}
    \begin{aligned}
      &\quad\text{ } \underset{\epsilon \to 0}{\mathop{\lim }}\,\frac{1}{\epsilon}\mathbb{E}\int^T_0\big(f(t,x^{\epsilon}_t,y^{\epsilon}_t,z^{\epsilon}_t,u^{\epsilon}_t)-f(t,\bar x_t,\bar y_t,\bar z_t,\bar u_t)\big)\mathrm{d}t \\
      &= \mathbb{E}\int^T_0\big((f^x_t)^{\top}x^1_t+(f^y_t)^{\top}y^1_t+(f^z_t)^{\top}z^1_t+(f^u_t)^{\top}\delta \bar u_t\big)\mathrm{d}t
    \end{aligned}
  \end{equation*}
  Therefore, it suffices to focus on the coupling term of the initial-terminal values $\Phi(x^{\epsilon}_T, y^{\epsilon}_0)-\Phi(\bar{x}_T, \bar{y}_0)$. We can obtain
  \begin{equation*}
    \begin{aligned}
      &\quad\text{ }\frac{1}{\epsilon}\mathbb{E}\big[\Phi(x^{\epsilon}_T,y^{\epsilon}_0)-\Phi(\bar x_T,\bar y_0)\big]\\
      &= \mathbb{E}\big[(\Phi^{x}_T)^{\top}x^1_T+(\Phi^{y}_0)^{\top}y^1_0\big]+\mathbb{E}\big[\Psi^{\epsilon}+\Xi^{\epsilon}\big],
    \end{aligned}
  \end{equation*}
  where 
  \begin{equation*}
    \left\{
      \begin{aligned}
        &\Phi^{\epsilon,x}_{T} = \int^1_0\Phi_x\big(\bar x_T+\lambda(x^{\epsilon}_T-\bar x_T),\bar y_0+\lambda(y^{\epsilon}_0-\bar y_0)\big)\mathrm{d}\lambda,\\
        &\Phi^{\epsilon,y}_{0} = \int^1_0\Phi_y\big(\bar x_T+\lambda(x^{\epsilon}_T-\bar x_T),\bar y_0+\lambda(y^{\epsilon}_0-\bar y_0)\big)\mathrm{d}\lambda,\\
        &\Psi^{\epsilon} = (\Phi^{\epsilon,x}_{T})^{\top}\Delta x^{\epsilon}_T+(\Phi^{\epsilon,y}_{0})^{\top}\Delta y^{\epsilon}_0,\\
        &\Xi^{\epsilon} = (\Phi^{\epsilon,x}_{T}-\Phi^{x}_T)^{\top}x^1_T+(\Phi^{\epsilon,y}_{0}-\Phi^{y}_0)^{\top}y^1_0.
      \end{aligned}
    \right.
  \end{equation*}
  By virtue of \Cref{le3_1} and dominated convergence theorem, $\underset{\epsilon \to 0}{\mathop{\lim }}\,\mathbb{E}|\Psi^{\epsilon}| = 0$ and $\underset{\epsilon \to 0}{\mathop{\lim }}\,\mathbb{E}|\Xi^{\epsilon}| = 0$ hold. Therefore, we can derive
  \begin{equation*}
    \underset{\epsilon \to 0}{\mathop{\lim }}\,\frac{1}{\epsilon}\mathbb{E}\big[\Phi(x^{\epsilon}_T,y^{\epsilon}_0)-\Phi(\bar x_T,\bar y_0)\big] = \mathbb{E}\big[(\Phi^{x}_T)^{\top}x^1_T+(\Phi^{y}_0)^{\top}y^1_0\big].
  \end{equation*}
  The proof is complete. 
\end{proof}
\subsection{Stochastic maximum principle}
\label{subsec3_2}
In this part, we derive the SMP for the optimal control problem with coupled initial-terminal values. To this end, we introduce the following adjoint equation:
\begin{equation}
  \label{eq009}
  \left\{
    \begin{aligned}
      &\mathrm{d}p_t = \big((g_t^y)^{\top}p_t-f_t^y\big)\mathrm{d}t+\big((g_t^z)^{\top}p_t-f_t^z\big)\mathrm{d}W_t,\\
      &\mathrm{d}q_t = -\big(\!-(g_t^x)^{\top}p_t\!+\!(b_t^x)^{\top}q_t\!+\!(\sigma_t^x)^{\top}\varpi_t\!+\!f^x_t\big)\mathrm{d}t\!+\!\varpi_t\mathrm{d}W_t,\\
      &p_0 = -\mathbb{E}[\Phi^y_0],\quad q_T = -(h^x_T)^{\top}p_T+\Phi^x_T.
    \end{aligned}
  \right.
\end{equation}
\begin{remark}
  \label{re3_1}
  It is worth noting that $p_0$ here is not equal to $-\Phi^y_0$. This is because the cost functional \eqref{eq007} contains a term $\Phi(x_T, y_0)$ that couples the terminal value $x_T$ of the forward process and the initial value $y_0$ of the backward process. Consequently, $\Phi^y_0 = \Phi_{y}(\bar x_T, \bar y_0)$ is not $\mathcal{F}_0$-measurable, which leads to $p_0 = -\mathbb{E}[\Phi^y_0]$ in \eqref{eq009}. If $x_T$ and $y_0$ are decoupled, $p_0$ reduces to an expression without the mathematical expectation (see, for example, (4.11) in \cite{P1993} and (3.7) in \cite{W1998}).
\end{remark}

Under (H1)-(H2), \eqref{eq009} admits a unique solution $(p_{\cdot},q_{\cdot},\varpi_{\cdot})$ $\in S^2_{\mathbb{F}}(0,T;\mathbb{R}^{m})\times S^2_{\mathbb{F}}(0,T;\mathbb{R}^{n})\times L^2_{\mathbb{F}}(0,T;\mathbb{R}^{n\times d})$. To proceed, we define the following Hamiltonian function $H:\Omega \times [0, T] \times \mathbb{R}^n \times \mathbb{R}^m \times \mathbb{R}^{m\times d}  \times  \mathbb{R}^m \times \mathbb{R}^n \times \mathbb{R}^{n\times d} \times U \rightarrow \mathbb{R}$,
\begin{equation}
  \label{eq010}
  \begin{aligned}
    &H(t,x,y,z,p,q,\varpi,u) = \langle q, b(t,x,u) \rangle + \langle \varpi, \sigma(t,x,u) \rangle \\
    &\qquad\qquad\qquad\qquad - \langle p, g(t,x,y,z,u) \rangle+f(t,x,y,z,u).
  \end{aligned}
\end{equation}
Then, we state the SMP for optimal control of our Problem (FBIT).
\begin{theorem}
  \label{th3_1}
  Let (H1)-(H2) hold. Suppose that $\bar{u}_{\cdot}\in {{\mathcal{U}}_{ad}}$ is an optimal control of Problem (FBIT) and $(\bar{x}_{\cdot}, \bar{y}_{\cdot}, \bar{z}_{\cdot})$ is the optimal trajectory of \eqref{eq006}. $(p_{\cdot},q_{\cdot},\varpi_{\cdot})$ is solution of \eqref{eq009} and $H$ satisfies \eqref{eq010}. Then for any $v\in U$, we have
  \begin{equation}
    \label{eq011}
    \langle H_u(t,\bar x_t,\bar y_t,\bar z_t, p_t, q_t, \varpi_t, \bar u_t), (v-\bar u_t) \rangle \geq 0, \text{ a.e., a.s.. }
  \end{equation}
\end{theorem}
\begin{proof}
  Applying It\^o's formula to $\langle p_t \cdot y^1_t \rangle+\langle q_t \cdot x^1_t \rangle$ and using Fubini's theorem, we have
\begin{equation}
  \label{eq012}
  \begin{aligned}
    &\quad\text{ }\mathbb{E}\bigg[\int^T_0\big((f^x_t)^{\top}x^1_t+(f^y_t)^{\top}y^1_t+(f^z_t)^{\top}z^1_t\big)\mathrm{d}t\\
    &\qquad\quad+(\Phi^{x}_T)^{\top}x^1_T+(\Phi^{y}_0)^{\top}y^1_0\bigg]\\
    &=\mathbb{E}\int^T_0\big(q_t^{\top}b_t^u\delta \bar u_t+\varpi_t^{\top}\sigma_t^u\delta \bar u_t-p_t^{\top}g_t^u\delta \bar u_t\big)\mathrm{d}t.
  \end{aligned}
\end{equation}
Combining \eqref{eq010} and \eqref{eq012}, an application of \Cref{le3_2} yields
\begin{equation*}
  \mathbb{E}\int^T_0\langle H_u(t,\bar x_t,\bar y_t,\bar z_t, p_t, q_t, \varpi_t, \bar u_t), \delta \bar u_t \rangle\mathrm{d}t\geq 0.
\end{equation*}
Thus \eqref{eq011} follows. The proof is complete. 
\end{proof}

Next, we can also prove that \eqref{eq011} is the sufficient condition for the optimal control under some conditions.

\begin{theorem}
  \label{th3_2}
  Let (H1)-(H2) hold.  Suppose that $\Phi, H$ are convex with respect to $(x,y,z,u)$ and $h$ is linear with respect to $x$. Let $\bar u_{\cdot} \in {{\mathcal{U}}_{ad}}$ be an admissible control. $(\bar x_{\cdot},\bar y_{\cdot},\bar z_{\cdot})$ and $(p,q,\varpi)$ are the corresponding solution of \eqref{eq006} and \eqref{eq009}, respectively. Then, if \eqref{eq011} is satisfied, $\bar u_{\cdot}$ is an optimal control of Problem (FBIT).
\end{theorem}
\begin{proof}
  Let $\hat u_{\cdot}$ be an admissible control and $(\hat x_{\cdot},\hat y_{\cdot},\hat z_{\cdot})$ be the corresponding solution. For simplicity, we denote $H_{\cdot}^{\tau}={{H}_{\tau}}(\cdot,{{\bar{x}}_{\cdot}},{{\bar{y}}_{\cdot}},{{\bar{z}}_{\cdot}},{{p}_{\cdot}},{{q}_{\cdot}},{{\varpi}_{\cdot}},{{\bar{u}}_{\cdot}})$, $\Delta {{H}_{\cdot}}=H(\cdot,{{\bar{x}}_{\cdot}},{{\bar{y}}_{\cdot}},{{\bar{z}}_{\cdot}},{{p}_{\cdot}},{{q}_{\cdot}},{{\varpi}_{\cdot}},$ ${{\bar{u}}_{\cdot}})-H(\cdot,{{\hat{x}}_{\cdot}},{{\hat{y}}_{\cdot}},{{\hat{z}}_{\cdot}},{{p}_{\cdot}},{{q}_{\cdot}},{{\varpi}_{\cdot}},{{\hat{u}}_{\cdot}})$, $\Delta {{f}_{\cdot}} = f(\cdot,{{\bar{x}}_{\cdot}},{{\bar{y}}_{\cdot}},{{\bar{z}}_{\cdot}},{{\bar{u}}_{\cdot}})-f(\cdot,{{\hat{x}}_{\cdot}},{{\hat{y}}_{\cdot}},{{\hat{z}}_{\cdot}},{{\hat{u}}_{\cdot}})$, $\Delta {{h}_{T}} = h({{\bar{x}}_T})-h({{\hat{x}}_T})$, $\Delta \tau =\bar\tau-\hat \tau$ with $\tau=x,y,z,u$. Then, applying It\^o's formula to $\langle{{p}_{t}},\Delta y_t\rangle+\langle{{q}_{t}},\Delta x_t\rangle$, we have
\begin{equation*}
  \begin{aligned}
    &\quad\text{ }\mathbb{E}\Big[(\Phi^x_T)^{\top}\Delta x_T+(\Phi^y_0)^{\top}\Delta y_0+p^{\top}_T\big(\Delta h_T-h^x_T\Delta x_T\big)\Big]\\
    &=\mathbb{E}\int^T_0\big(\Delta H_t - \langle H^x_t, \Delta x_t \rangle - \langle H^y_t, \Delta y_t \rangle \\
    &\qquad\qquad\quad- \langle H^z_t, \Delta z_t \rangle-\Delta f_t \big)\mathrm{d}t.
  \end{aligned}
\end{equation*}
Note that $\Phi, H$ are convex and $h$ is linear. As a consequence,
\begin{equation*}
  \begin{aligned}
    &\quad\text{ }J(\bar u_{\cdot})-J(\hat u_{\cdot})\\
    &\leq \mathbb{E}\int^T_0\big(\Delta H_t - \langle H^x_t, \Delta x_t \rangle - \langle H^y_t, \Delta y_t \rangle \\
    &\qquad\qquad\quad- \langle H^z_t, \Delta z_t \rangle-\langle H^u_t, \Delta u_t \rangle\big)\mathrm{d}t\\
    &\leq 0.
  \end{aligned}
\end{equation*}
The proof is complete.
\end{proof}

\section{Efficient strategies} \label{sec3}
In this section, we apply the results established in \Cref{sec2} to the mean-variance portfolio selection problem proposed in \Cref{sec1} (i.e., to solve \eqref{eq005}). Before that, we consider a more general stochastic LQ control problem. Its cost functional not only involves a coupling term between the terminal and initial values but includes a running cost. Furthermore, the aforementioned portfolio selection problem, is obtained as a special case of this generalized formulation.

\subsection{A general stochastic LQ problem}
\label{subsec4_1}
Consider the following SDE system:
\begin{equation}
  \label{eq013}
  \left\{
    \begin{aligned}
      &\mathrm{d}x_t = (A_tx_t+B_tu_t)\mathrm{d}t+u^{\top}_tD_t\mathrm{d}W_t,\\
      &x_0=a,
    \end{aligned}
  \right.
\end{equation}
where $a\in \mathbb{R}$, $A:[0,T]\to \mathbb{R}$, $B:[0,T]\to \mathbb{R}^{1\times d}$, $D:[0,T]\to \mathbb{R}^{d\times d}$, and $A_{\cdot}, B_{\cdot}, D_{\cdot}$ are uniformly bounded measurable functions. Define the admissible control set $\mathcal U_{ad}:=\{u| u_{\cdot} \in L^2_{\mathbb{F}}(0,T;\mathbb{R}^d)\}$. Now, we introduce the cost functional:
\begin{equation}
  \label{eq014}
  \begin{aligned}
    &J(u_{\cdot}) = \mathbb{E}\bigg[\int^T_0\big(L_tx^2_t+\langle R_tu_t, u_t \rangle\big)\mathrm{d}t\\
    &\qquad\qquad\quad+\big(Qx_T+S\mathbb{E}x_T\big)^2+Vx_T\bigg],
  \end{aligned}
\end{equation}
where $Q, S, V\in \mathbb{R}$, $L:[0,T]\to \mathbb{R}$, $R:[0,T]\to \mathbb{S}^{d}$, and $L_{\cdot}, R_{\cdot}$ are uniformly bounded measurable functions, and $R_{\cdot}+Q^2D_{\cdot}D_{\cdot}^{\top}\gg 0$.
\begin{remark}
  \label{re4_1}
  It is worth noting that the cost functional \eqref{eq014} not only involves a coupling term between the terminal and initial values but also includes a running cost. Furthermore, we relax the uniformly positive definite condition typically required for the running cost $R_{\cdot}$ in classical LQ problems. Consequently, the conditions can be satisfied even when $R_{\cdot}$ is indefinite, which facilitates the investigation of the portfolio selection problem in the following subsections.
\end{remark}

Applying It\^o's formula to $x^2_t$, one can show that
\begin{equation}
  \label{eq015}
  \mathbb{E}\big[x^2_T\big] = a^2+ \mathbb{E}\int^T_0\big(2A_tx_t^2+2x_tB_tu_t+\langle u_t^{\top}D_t, u_t^{\top}D_t \rangle\big)\mathrm{d}t.
\end{equation}
Combining \eqref{eq014} and \eqref{eq015}, we have
\begin{equation}
  \label{eq016}
  \begin{aligned}
    &J(u_{\cdot}) = \mathbb{E}\bigg[\int^T_0\big(\tilde L_tx^2_t+2Q^2x_tB_tu_t+\langle \tilde R_tu_t, u_t \rangle\big)\mathrm{d}t\\
    &\qquad\qquad\quad+2QSx_T\mathbb{E}x_T+S^2\big(\mathbb{E}x_T\big)^2+Vx_T+Q^2a^2\bigg],
  \end{aligned}
\end{equation}
where $\tilde L_{\cdot} = L_{\cdot} + 2Q^2A_{\cdot}$ and $\tilde R_{\cdot} = R_{\cdot}+Q^2D_{\cdot}D_{\cdot}^{\top}$. By adopting the same transformation technique as in \eqref{eq005}, we introduce the stochastic process $y_t := \mathbb{E}[x_T | \mathcal{F}_t]$ for $0 \leq t \leq T$. Consequently, by incorporating \eqref{eq013} and \eqref{eq016}, we can rewrite the state equation and the cost functional, respectively, as follows:
\begin{equation*}
  \left\{
    \begin{aligned}
      &\mathrm{d}x_t = (A_tx_t+B_tu_t)\mathrm{d}t+u^{\top}_tD_t\mathrm{d}W_t,\\
      &\mathrm{d}y_t = z_t\mathrm{d}W_t,\\
      &x_0=a,\quad y_T=x_T,
    \end{aligned}
  \right.
\end{equation*}
and
\begin{equation*}
  \begin{aligned}
    &J(u_{\cdot}) = \mathbb{E}\bigg[\int^T_0\big(\tilde L_tx^2_t+2Q^2x_tB_tu_t+\langle \tilde R_tu_t, u_t \rangle\big)\mathrm{d}t\\
    &\qquad\qquad\quad+2QSx_Ty_0+S^2y_0^2+Vx_T+Q^2a^2\bigg].
  \end{aligned}
\end{equation*}
By virtue of \Cref{th3_1} and \Cref{th3_2}, it follows that the optimal control $\bar{u}_{\cdot}$ can be obtained from \eqref{eq011}. Let $(\bar{x}_{\cdot}, \bar{y}_{\cdot}, \bar{z}_{\cdot})$ be the state trajectories associated with $\bar{u}_{\cdot}$. Then, we can derive
\begin{equation}
  \label{eq017}
  B^{\top}_tq_t+D_t\varpi^{\top}_t+2Q^2B^{\top}_t\bar x_t+2\tilde{R}_t\bar u_t = 0, \text{ a.e., a.s.}, 
\end{equation}
where
\begin{equation}
  \label{eq018}
  \left\{
    \begin{aligned}
      &\mathrm{d}q_t = -\big(A_tq_t+2\tilde L_t\bar x_t+2Q^2B_t\bar u_t\big)\mathrm{d}t+\varpi_t \mathrm{d}W_t,\\
      &q_T = 2S(2Q+S)\bar y_0+V.
    \end{aligned}
  \right.
\end{equation}
Next, in order to obtain the feedback representation of the optimal control $\bar{u}_{\cdot}$, we introduce the following equations:
\begin{equation}
  \label{eq019}
  \left\{
    \begin{aligned}
      &\dot \Pi_t +2A_t\Pi_t - (\Pi_t+2Q)(\Pi_t+2Q^2)B_t\Lambda_tB_t^{\top}+2\tilde L_t = 0,\\
      &\Pi_T = 0,
    \end{aligned}
  \right.
\end{equation}
\begin{equation}
  \label{eq020}
  \left\{
    \begin{aligned}
      &\dot \Sigma_t \!+\!2\big[A_t-(\Pi_t+Q^2+Q)B_t\Lambda_tB_t^{\top}\big]\Sigma_t - B_t\Lambda_tB_t^{\top}\Sigma^2_t = 0,\\
      &\Sigma_T = 2S(2Q+S),
    \end{aligned}
  \right.
\end{equation}
and
\begin{equation}
  \label{eq021}
  \left\{
    \begin{aligned}
      &\dot \rho_t +\big[A_t-(\Pi_t+\Sigma_t+2Q)B_t\Lambda_tB_t^{\top}\big]\rho_t= 0,\\
      &\rho_T = V,
    \end{aligned}
  \right.
\end{equation}
where $\Gamma_{\cdot} = 1 + \frac{1}{2}\Pi_{\cdot}D^{\top}_{\cdot}\tilde R_{\cdot}^{-1}D_{\cdot}$ and $\Lambda_{\cdot}=\frac{1}{2}\tilde R_{\cdot}^{-1}-\frac{1}{4}\Pi_{\cdot}\tilde R_{\cdot}^{-1}D_{\cdot}\Gamma_{\cdot}^{-1}D^{\top}_{\cdot}\tilde R_{\cdot}^{-1}$. Now, we derive the feedback representation of the optimal control $\bar{u}_{\cdot}$.
\begin{proposition}
  \label{pr4_1}
  If \eqref{eq019} and \eqref{eq020} admit solutions $\Pi_{\cdot} \in C([0,T]; \mathbb{R})$ and $\Sigma_{\cdot} \in C([0,T]; \mathbb{R})$, respectively, then the optimal control $\bar u_{\cdot}$ has the following feedback representation:
  \begin{equation}
    \label{eq022}
    \bar u_t = -\Big[\big(\Pi_t+2Q^2\big)\bar x_t + \Sigma_t \mathbb{E}\bar x_t+\rho_t\Big]\Lambda_tB^{\top}_t, \text{ a.e., a.s.}.
  \end{equation}
\end{proposition}
\begin{proof}
  By combining \eqref{eq013} and \eqref{eq018}-\eqref{eq021}, and applying Itô's formula to $\Pi_t\bar x_t + \Sigma_t\mathbb{E}\bar x_t + \rho_t$, we can obtain
\begin{equation}
  \label{eq023}
  \left\{
    \begin{aligned}
      &q_t = \Pi_t\bar x_t + \Sigma_t\mathbb{E}\bar x_t + \rho_t,\\
      &\varpi_t = -\frac{1}{2}\Pi_t\Big[\big(\Pi_t+2Q^2\big)\bar x_t + \Sigma_t \mathbb{E}\bar x_t+\rho_t\Big]B_t\tilde R_t^{-1}D_t\Gamma^{-1}_t.
    \end{aligned}
  \right.
\end{equation}
Substituting \eqref{eq023} into \eqref{eq017}, \eqref{eq022} follows. The proof is complete.
\end{proof}
\begin{remark}
  \label{re4_2}
  If \eqref{eq019} admits a solution $\Pi_{\cdot}$, the explicit expressions for $\Sigma_{\cdot}$ and $\rho_{\cdot}$ can be readily derived since \eqref{eq020} and \eqref{eq021} are a Bernoulli equation and a linear ordinary differential equation, respectively. However, the presence of the term $(\Pi_t+2Q)(\Pi_t+2Q^2)B_t\Lambda_tB_t^{\top}$ in \eqref{eq019} leads to highly complex properties of its solution. Nevertheless, if the stochastic LQ problem \eqref{eq013}-\eqref{eq014} in this subsection degenerates into the mean-variance portfolio problem \eqref{eq003}-\eqref{eq004} in \Cref{sec1}, \eqref{eq019} reduces to a linear ordinary differential equation (see \eqref{eq026}), from which the explicit expression for $\Pi_{\cdot}$ can be directly obtained.
\end{remark}
\subsection{Mean-variance portfolio selection problem}
\label{subsec4_2}
We first cast the mean-variance portfolio selection problem from \Cref{sec1} into the framework of the stochastic LQ problem discussed in \Cref{subsec4_1}. By comparing \eqref{eq003}-\eqref{eq004} with \eqref{eq013}-\eqref{eq014}, one can show that
\begin{equation*}
  \left\{
    \begin{aligned}
      &A_{\cdot} \!=\! r_{\cdot},\text{ } B_{\cdot} \!=\! (b^1_{\cdot}-r_{\cdot},\cdots,b^d_{\cdot}-r_{\cdot}),D_{\cdot} \!= \left(\begin{matrix} \sigma^{11}_{\cdot}  & \cdots & \sigma^{1d}_{\cdot}\\\vdots  & \ddots &\vdots\\ \sigma^{d1}_{\cdot} & \cdots & \sigma^{dd}_{\cdot}\\ \end{matrix}\right),\\
      &L_{\cdot} = 0,\quad R_{\cdot} = {\bf 0}, \quad Q = 1, \quad S=-1, \quad V=-\lambda .
    \end{aligned}
  \right.
\end{equation*}
For simplicity, we denote $\tilde R_{\cdot}:=D_{\cdot}D^{\top}_{\cdot}$, $\varphi_{\cdot}:=B_{\cdot}(D^{\top}_{\cdot})^{-1}$ and $\theta_{\cdot}:= B_{\cdot}\tilde R^{-1}_{\cdot}B^{\top}_{\cdot}$. Based on \eqref{eq013} and \eqref{eq022}, we can further simplify the expressions to obtain:
\begin{equation}
  \label{eq024}
  \left\{
    \begin{aligned}
      &\mathrm{d}\bar x_t \!=\!\Big\{(A_t\!-\!\theta_t)\bar x_t\!-\!\frac{\Sigma_t}{\Pi_t+2}\theta_t\mathbb{E}\bar x_t\!-\!\frac{1}{\Pi_t+2}\theta_t\rho_t\Big\}\mathrm{d}t\\
      &\quad-\Big\{\varphi_t\bar x_t-\frac{\Sigma_t}{\Pi_t+2}\varphi_t\mathbb{E}\bar x_t-\frac{1}{\Pi_t+2}\varphi_t\rho_t\Big\}\mathrm{d}W_t,\\
      &\bar x_0=a,
    \end{aligned}
  \right.
\end{equation}
and
\begin{equation}
  \label{eq025}
  \bar u_t= -\tilde R^{-1}_tB^{\top}_t\left(\bar x_t+\frac{\Sigma_t}{\Pi_t+2}\mathbb{E}\bar x_t+\frac{1}{\Pi_t+2}\rho_t\right),
\end{equation}
where the equations corresponding to \eqref{eq019}-\eqref{eq021} are given by:
\begin{equation}
  \label{eq026}
  \left\{
    \begin{aligned}
      &\dot{\Pi_t}+(2A_t-\theta_t)\Pi_t-2\theta_t+4A_t = 0,\\
      &\Pi_T = 0,
    \end{aligned} 
  \right.
\end{equation}
\begin{equation}
  \label{eq027}
  \left\{
    \begin{aligned}
      &\dot{\Sigma_t}+2(A_t-\theta_t)\Sigma_t-\frac{1}{\Pi_t+2}\theta_t\Sigma_t^2 = 0,\\
      &\Sigma_T = -2,
    \end{aligned} 
  \right.
\end{equation}
and
\begin{equation}
  \label{eq028}
  \left\{
    \begin{aligned}
      &\dot \rho_t+\left[A_t-\left(1+\frac{\Sigma_t}{\Pi_t+2}\right)\theta_t\right]\rho_t=0,\\
      &\rho_T=-\lambda.
    \end{aligned} 
  \right.
\end{equation}
It is worth noting that \eqref{eq026} and \eqref{eq028} are linear ordinary differential equations, and \eqref{eq027} is a Bernoulli equation. Therefore, their explicit solutions can be obtained as follows:
\begin{equation}
  \label{eq029}
  \left\{
    \begin{aligned}
      &\Pi_t = 2e^{\int^T_t(2A_s-\theta_s)\mathrm{d}s}-2,\quad \Sigma_t = -2e^{\int^T_t(2A_s-\theta_s)\mathrm{d}s},\\
      &\rho_t = -\lambda e^{\int^T_tA_s\mathrm{d}s},\quad t\in [0,T].
    \end{aligned}
  \right.
\end{equation}
Substituting \eqref{eq029} into \eqref{eq024}, we derive
\begin{equation*}
  \left\{
    \begin{aligned}
      &\mathbb{E}\bar x_t = ae^{\int^t_0A_s\mathrm{d}s}+\frac{\lambda}{2}e^{\int^T_t(\theta_s-A_s)\mathrm{d}s}\left(e^{\int^t_0\theta_s\mathrm{d}s}-1\right),\\
      &Var[\bar x_t] \!=\! \mathbb{E}[(\bar x_t-\mathbb{E}\bar x_t)^2] \!=\! \frac{\lambda^2}{4}e^{2\int^T_t(\theta_s-A_s)\mathrm{d}s}\!\left(e^{\int^t_0\theta_s\mathrm{d}s}\!-1\!\right).
    \end{aligned}
  \right.
\end{equation*}
Thus, the efficient frontier is 
\begin{equation*}
  Var[\bar x_T] = \frac{(\mathbb{E}\bar x_T-ae^{\int^T_0A_t\mathrm{d}t})^2}{e^{\int^T_0\theta_t\mathrm{d}t}-1}.
\end{equation*}
\begin{remark}
  \label{re4_3}
  The efficient frontier in the mean-standard-deviation diagram is still a straight line. To be specific, we have
  \begin{equation*}
    \mathbb{E}\bar x_T = ae^{\int^T_0A_t\mathrm{d}t} + \sqrt{e^{\int^T_0\theta_t\mathrm{d}t}-1}\sqrt{Var[\bar x_T]},
  \end{equation*}
  which is also called the capital market line. 
\end{remark}

Furthermore, by virtue of \eqref{eq024}, \eqref{eq025}, and \eqref{eq029}, we can obtain that the explicit expression 
for the investor's optimal total wealth $\bar{x}_{\cdot}$ is given by 
\begin{equation}
  \label{eq30}
  \begin{aligned}
    &\bar x_t = ae^{\int^t_0A_s\mathrm{d}s}+\frac{\lambda}{2}e^{\int^T_t(\theta_s-A_s)\mathrm{d}s}\\
    &\quad\text{ }\text{ }\times \bigg\{e^{\int^t_0\theta_s\mathrm{d}s}+\int^t_0\left[e^{-\frac{1}{2}\int^t_s\theta_r\mathrm{d}r-\int^t_s\varphi_r\mathrm{d}W_r}\theta_s\right]\mathrm{d}s \\
    &\quad\text{ }\text{ }+\!e^{-\!\int^t_0\!\varphi_s\mathrm{d}W_s}\!\int^t_0\!\left[e^{-\frac{1}{2}\int^t_s\theta_r\mathrm{d}r+\int^s_0\varphi_r\mathrm{d}W_r}\varphi_s\right]\!\mathrm{d}W_s\!-\!1\!\bigg\}, 
  \end{aligned}
\end{equation}
and the feedback representation for the optimal portfolio strategy $\bar{u}_{\cdot}$ is given by 
\begin{equation}
  \label{eq31}
  \bar u_t=-\tilde R_t^{-1}B_t^{\top}\left(\bar x_t-\mathbb{E}\bar x_t-\frac{\lambda}{2}e^{\int_t^T(\theta_s-A_s)\mathrm{d}s}\right).
\end{equation}

\begin{figure*}[t]
\centering
\includegraphics[width=0.9\textwidth]{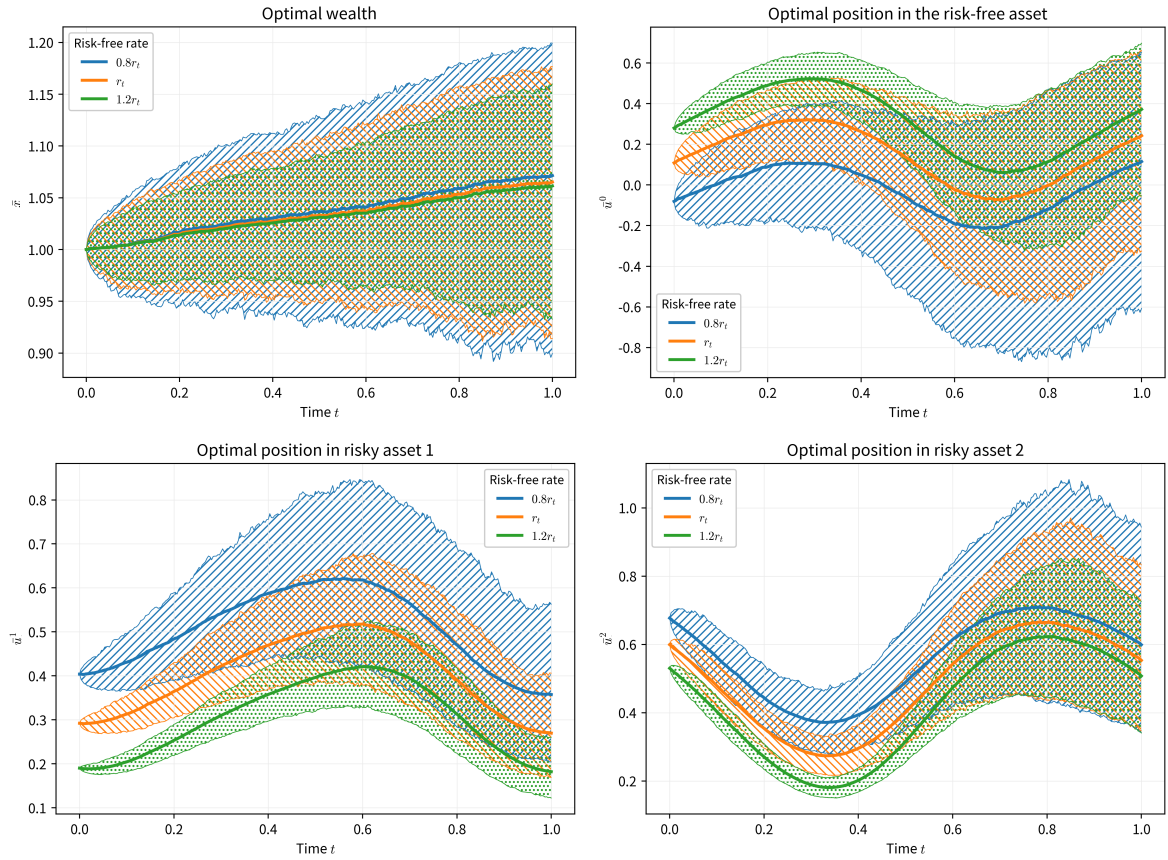}
\caption{Effect of the risk-free interest rate on optimal wealth and investment strategies: based on the parameters in \eqref{eq32}, the rate is set to $0.8r_t$, $r_t$, and $1.2r_t$, with all other parameters unchanged; solid lines and shaded areas represent sample means and $10$th--$90$th percentile bands, respectively.}
\label{fig1}
\end{figure*}

\begin{figure*}[t]
\centering
\includegraphics[width=0.81\textwidth]{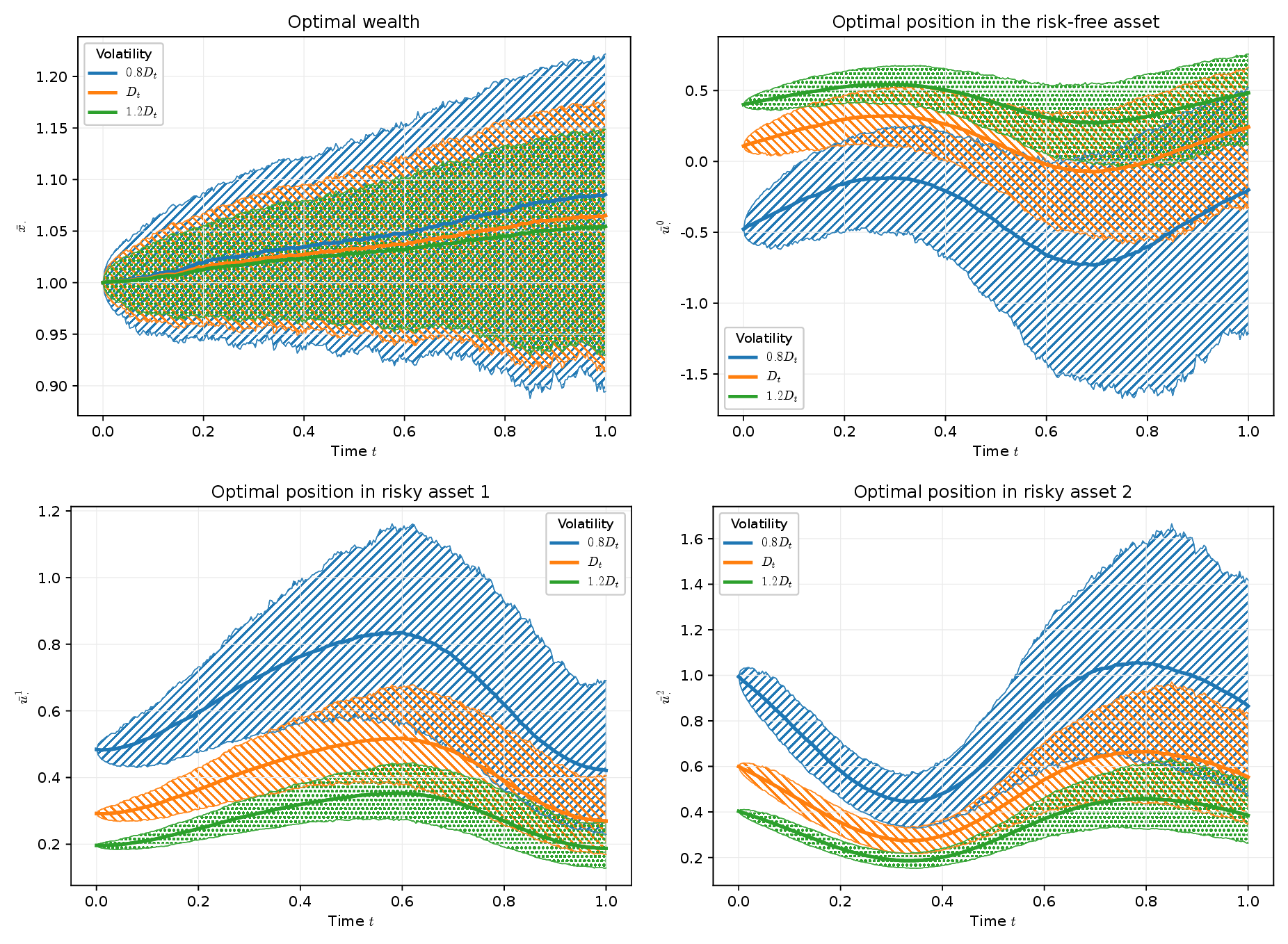}
\caption{Effect of volatility on optimal wealth and investment strategies: based on the parameters in \eqref{eq32}, the volatility matrix is set to $0.8D_t$, $D_t$, and $1.2D_t$, with all other parameters unchanged; solid lines and shaded areas represent sample means and $10$th--$90$th percentile bands, respectively.}
\label{fig2}
\end{figure*}

\begin{figure*}[t]
\centering
\includegraphics[width=0.81\textwidth]{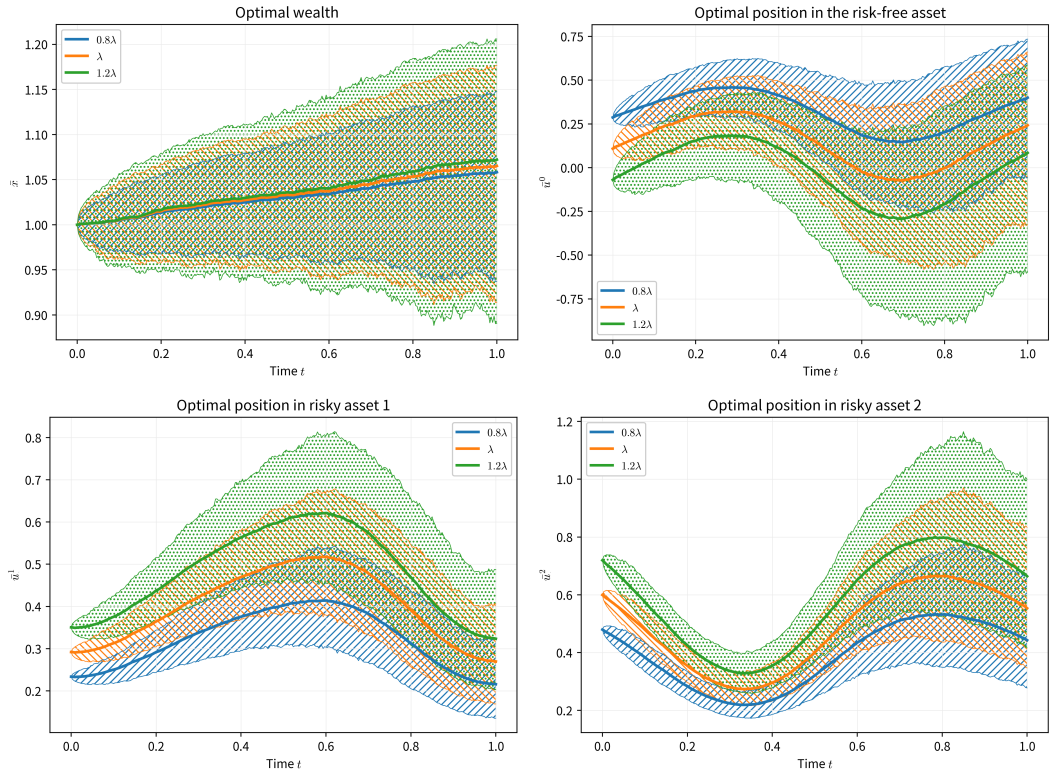}
\caption{Effect of $\lambda$ on optimal wealth and investment strategies: based on the parameters in \eqref{eq32}, the parameter is set to $0.8\lambda$, $\lambda$, and $1.2\lambda$, with all other parameters unchanged; solid lines and shaded areas represent sample means and $10$th--$90$th percentile bands, respectively.}
\label{fig3}
\end{figure*}

To better illustrate our results, we present some numerical simulations below. Suppose that
\begin{equation}
  \label{eq32}
  \left\{
    \begin{aligned}
      & T=1,\quad a=1,\quad d=2,\quad \lambda=0.6,\\
      & A_t = r_t = 0.03+0.01\sin(2\pi t),\\
      & B_t = (b_t^1,\text{ } b_t^2) \\
      & \quad\text{ } = (0.06+0.01\sin(2\pi t),\text{ } 0.07+0.01\cos(2\pi t)),\\
      & D_t = \begin{pmatrix} \sigma^{11}_t & \sigma^{12}_t\\ \sigma^{21}_t & \sigma^{22}_t \end{pmatrix}\\
      & \quad\text{ } = \begin{pmatrix} 0.12+0.01\sin(2\pi t) & 0.03+0.01\cos(2\pi t)\\ 0.02+0.01\sin(2\pi t) & 0.14+0.01\cos(2\pi t) \end{pmatrix}\!.
    \end{aligned}
  \right.
\end{equation}
Based on \eqref{eq30} and \eqref{eq31}, we generate sample paths using the Monte Carlo method to examine the sensitivity of optimal wealth and investment strategies to the risk-free interest rate, volatility, and the parameter $\lambda$. \Cref{fig1} shows that an increase in the risk-free interest rate reduces the mean optimal positions in the two risky assets and increases the mean position in the risk-free asset, while slightly lowering the sample mean of optimal terminal wealth. This indicates that, when the expected returns of the risky assets remain unchanged, a higher interest rate reduces risk premia and prompts investors to reduce their allocations to risky assets. \Cref{fig2} shows that, as volatility increases, the mean optimal positions in the two risky assets and the sample mean of optimal wealth decrease, whereas the mean position in the risk-free asset increases and the $10$th--$90$th percentile bands for wealth and all asset positions narrow. This indicates that, in response to higher risky-asset volatility, investors typically reduce their risky-asset holdings and increase their allocation to the risk-free asset, thereby reducing the volatility of total wealth. The parameter $\lambda$ reflects the importance investors attach to expected returns. \Cref{fig3} shows that, as $\lambda$ increases, the mean optimal positions in the two risky assets and the sample mean of optimal wealth increase, whereas the mean position in the risk-free asset decreases and the $10$th--$90$th percentile bands for wealth and all asset positions widen. Thus, investors must accept greater wealth fluctuations to achieve higher expected returns. The capital market line in \Cref{re4_3} likewise indicates a positive relationship between expected wealth and risk.

\section{Conclusion}\label{sec4}
Starting from the mean-variance portfolio selection problem, this paper establishes a connection between mean-field systems and forward-backward stochastic systems, and formulates a class of control problems for forward-backward systems with coupled initial-terminal values. For this coupling structure, we establish an SMP for control problems with a running cost and coupled initial-terminal terms, and provide sufficient conditions for optimality under appropriate convexity assumptions. Furthermore, we study a class of generalized stochastic LQ control problems, allowing the control weight in the running cost to be indefinite, and derive a feedback representation of the optimal control, with the classical mean-variance model included as a special case of this framework. The corresponding efficient investment strategy and efficient frontier are characterized explicitly, and numerical simulations further illustrate the dynamic features of the optimal wealth and asset allocation, as well as their responses to changes in interest rates, volatility, and $\lambda$. Future work may investigate the connection between conditional mean-field systems and forward-backward stochastic systems, thereby extending the present framework to practical financial problems that can be formulated as conditional mean-field control problems under partial information or in more general settings.

\printcredits

\section*{Declaration of competing interest}
The authors declare that they have no known competing financial interests or personal relationships that could have appeared to influence the work reported in this paper.

\section*{Data availability}
  No data were used for the research described in the article.

\section*{Acknowledgments}
The authors are grateful for the National Key R\&D Program of China [grant number 2023YFA1009200]; the Natural Science Foundation of China [grant numbers 125B2025, 12521001, 62561160159]; and the Taishan Scholars Climbing Program of Shandong [grant number TSPD20210302].

\end{document}